[1994/12/01]% LaTeX date must December 1994 or later
\documentclass{ijmart-mod}
\chardef\bslash=`\\ % p. 424, TeXbook
\usepackage{graphicx}
\usepackage[breaklinks=true]{hyperref}
\usepackage{mathtools}

\DeclareSymbolFont{symbolsC}{U}{txsyc}{m}{n}
\DeclareMathSymbol{\varparallelinv}{\mathrel}{symbolsC}{10}

\newtheorem{thm}{Theorem}[section]

\newtheorem{lem}[thm]{Lemma}
\newtheorem{prop}[thm]{Proposition}

\theoremstyle{definition}
\newtheorem{defn}[thm]{Definition}
\newtheorem{rem}[thm]{Remark}

\theoremstyle{remark}

\newcommand{\contract}{\mathord{\varparallelinv}} 

\newcommand{\eval}[2][\right]{\relax
  \ifx#1\right\relax \left.\fi#2#1\rvert}

\begin{document}

\title{Counting geodesics between surface~triangulations}

\author[H. Parlier]{Hugo Parlier}
\address{Department of Mathematics, FSTM,\\University of Luxembourg,\\Esch-sur-Alzette, Luxembourg}
\email{hugo.parlier@uni.lu}

\author[L. Pournin]{Lionel Pournin}
\address{Universit{\'e} Paris 13, Villetaneuse, France}
\email{lionel.pournin@univ-paris13.fr}

%\date{Received on MONTH, YEAR}
%\issueinfo{VOL}{NUM}{MONTH}{YEAR}
%\doiinfo{10.1007/DOI-NUMBER}
\begin{abstract}
Given a surface $\Sigma$ equipped with a set $P$ of marked points, we consider the triangulations of $\Sigma$ with vertex set $P$. The flip-graph of $\Sigma$ whose vertices are these triangulations, and whose edges correspond to flipping arcs appears in the study of moduli spaces and mapping class groups. We consider the number of geodesics in the flip-graph of $\Sigma$ between two triangulations as a function of their distance. We show that this number grows exponentially provided the surface has enough topology, and that in the remaining cases the growth is polynomial.
\end{abstract}
\maketitle
%\tableofcontents

\section{Introduction}\label{PP4.sec.0}

The number of geodesics between two points in a metric space appears in many contexts, and is related to understanding the dynamical, geometric and topological properties of this space. In particular, for compact manifolds, the problem of counting closed geodesics has many facets and is related to the volume entropy. These notions enjoy counterparts in combinatorial settings and in particular for graphs \cite{Balacheff2007}. Certain manifolds, such as the standard round sphere or Euclidean spaces, enjoy ``blocking'' properties, meaning that removing a small number of points (or cutting away a small volume submanifold) can block all the geodesics between two points (see for instance \cite{GutkinSchroeder2006, SchmidtSouto2010}). The unicity, or local unicity, of geodesics between points is a feature of non-positive curvature, but it can fail even in coarse settings. The curve graph is a relevant example: while it is Gromov hyperbolic \cite{MasurMinsky1999}, there are infinitely-many geodesic paths between pairs of vertices at distance $2$. In the related pants graph however \cite{AramayonaParlierShackleton2008}, it is still unclear whether a similar property holds.

Flip-graphs are naturally associated to finite-type surfaces, just like curve graphs or pants graphs, and we explore the problem of counting the geodesics between their vertices.
More precisely, we consider an orientable surface $\Sigma$ of genus $g$, with $b$ boundary curves, and $p$ punctures. We equip $\Sigma$ with finitely-many marked points: each puncture of $\Sigma$ is a marked point and all the other marked points of $\Sigma$ are placed in its boundary curves in such a way that each boundary curve of $\Sigma$ contains at least one marked point. A triangulation of $\Sigma$ can be thought of as a set of pairwise non-crossing arcs between the marked points of $\Sigma$ that decompose $\Sigma$ into triangles (see Section \ref{PP4.sec.1} for a precise definition). The flip-graph of $\Sigma$ is the graph $\mathcal{F}(\Sigma)$ whose vertices are the triangulations of $\Sigma$ and whose edges connect two triangulations that differ by exactly one arc. We think of $\mathcal{F}(\Sigma)$ as a metric graph whose edges have length $1$.  

The vertex degrees of $\mathcal{F}(\Sigma)$ are bounded and with the exception of certain topological types, $\mathcal{F}(\Sigma)$ is infinite. These flip-graphs are related to the study of moduli spaces and mapping class groups~\cite{DisarloParlier2019}. Polygons form a prominent example of surfaces whose flip-graphs are finite. In this case, the marked points are the vertices of the polygon. The resulting flip-graphs are the $1$-skeletons of associahedra~\cite{Lee1989,Stasheff1963a,Stasheff1963b,Tamari1951} and their geometric properties are well-studied \cite{Pournin2014,PourninWang2021,SleatorTarjanThurston1988}. The punctured disks with marked points on the boundary are another example of surfaces whose flip-graphs are finite \cite{ParlierPournin2018b}. The only other surface whose flip-graph is finite is the $3$-punctured sphere (that flip-graph is depicted in Figure \ref{PP4.sec.1.fig.1}). Here, we are primarily interested in infinite flip-graphs, whose lengths of geodesics can get arbitrarily large.

The quantity we study in this article is the largest number $\Delta_k(\Sigma)$ of geodesics in $\mathcal{F}(\Sigma)$ between any two vertices at distance $k$. The vertex degrees of $\mathcal{F}(\Sigma)$ being uniformly bounded implies that $\Delta_k(\Sigma)$ is always finite. More precisely, there is an immediate upper bound on $\Delta_k(\Sigma)$ which grows exponentially as a function of $k$ (with the degree minus $1$ as a basis for the exponential). Our first theorem identifies when there is a matching exponential lower bound. 

\begin{thm}\label{PP4.sec.0.thm.1}
If $b$ and $p$ are not both equal to $0$ and if $2g+p+b$ is at least $3$, then the growth of $\Delta_k(\Sigma)$ as a function of $k$ is at least exponential as $k$ tends to infinity, except possibly when
\begin{itemize}
\item[(i)] $\Sigma$ is the $4$-punctured sphere ($g=b=0$ and $p=4$),
\item[(ii)] $\Sigma$ is a genus one surface with a unique boundary component and no puncture ($g=b=1$ and $p=0$), or
\item[(iii)] $\Sigma$ is a $2$-punctured disk ($g=0$, $b=1$, and $p=2$).
\end{itemize}
\end{thm}

We also investigate the cases that are not covered by Theorem~\ref{PP4.sec.0.thm.1}. Figure~\ref{PP4.sec.1.fig.2} provides a visual diagram of when the exponential behavior kicks in for the low complexity cases. As will be discussed in the sequel, some cases are very easy (namely when the graphs are finite) and in a select few number of cases, we were unable to determine the growth rate with our methods. 

\begin{figure}[b]
\begin{centering}
\includegraphics[scale=1]{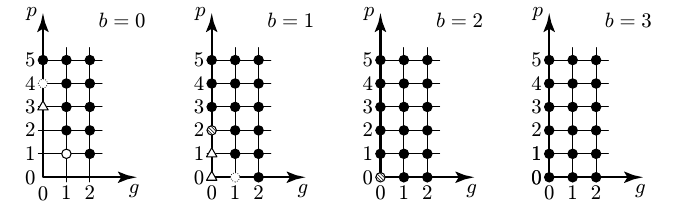}
\caption{The behavior of $\Delta_k(\Sigma)$ when $b\leq3$, $g\leq2$ and $p\leq5$. Symbols indicate the behavior of $\Delta_k(\Sigma)$ as $k$ goes to infinity: solid disks indicate exponential behavior, striped disks polynomial behavior, solid circles that $\Delta_k(\Sigma)=1$, and dotted circles represent unsolved cases. Triangles mean that $\mathcal{F}(\Sigma)$ is finite. When $b=0$ and there is no symbol, $\mathcal{F}(\Sigma)$ is empty.}\label{PP4.sec.1.fig.2}
\end{centering}
\end{figure}

When $\Sigma$ is a cylinder without punctures, we show the following.
\begin{thm}\label{PP4.sec.0.thm.2}
If $\Sigma$ is a cylinder without punctures, then the growth of $\Delta_k(\Sigma)$ as a function of $k$ is at most polynomial as $k$ goes to infinity.
\end{thm}

Note that, when $\Sigma$ is a cylinder without punctures, $b=2$ and $g=p=0$ and thus $2g+p+b$ is at most $2$. This explains why this case is not listed as an exception in the statement of Theorem \ref{PP4.sec.0.thm.1}.

We also study the case of a doubly punctured disk.

\begin{thm}\label{PP4.sec.0.thm.3}
If $\Sigma$ is a $2$-punctured disk, then the growth of $\Delta_k(\Sigma)$ as a function of $k$ is at most polynomial as $k$ goes to infinity.
\end{thm}

In fact Theorems \ref{PP4.sec.0.thm.2} and \ref{PP4.sec.0.thm.3} will be obtained as a consequence of a more general result. We will show that, as $k$ goes to infinity, $\Delta_k(\Sigma)$ depends only polynomially on two quantities related to the surface $\Sigma^\star$ obtained from $\Sigma$ by removing all but one marked points on each boundary curve. The first of these quantities, $\widetilde{\Delta}_k(\Sigma^\star)$ is the maximum value of $\Delta_i(\Sigma^\star)$ when $i$ ranges between $1$ and $k$. The second quantity, $\Lambda_k(\Sigma^\star)$ is the maximum number of triangulations in any ball of radius $k$ in $\mathcal{F}(\Sigma^\star)$.

\begin{thm}\label{PP4.sec.0.thm.4}
When $k$ goes to infinity, $\Delta_k(\Sigma)$ is upper bounded by a polynomial function of $\widetilde{\Delta}_k(\Sigma^\star)$ and $\Lambda_k(\Sigma^\star)$. 
\end{thm}

The bounds on $\Delta_k(\Sigma)$ stated by Theorems \ref{PP4.sec.0.thm.1}, \ref{PP4.sec.0.thm.2}, and \ref{PP4.sec.0.thm.4} will be given explicitly (see Theorems \ref{PP4.sec.2.thm.1}, \ref{PP4.sec.1.5.thm.1}, and \ref{PP4.sec.3.thm.1}). 

Finally, we point out that there are only two remaining cases, for which we do not know whether $\Delta_k(\Sigma)$ grows exponentially or polynomially as a function of $k$. These two cases are the $4$-punctured sphere and the genus one surface with a single boundary component and no puncture.

\medskip

\noindent{\bf Acknowledgments.}
Hugo Parlier is supported by the Luxembourg National Research Fund OPEN grant O19/13865598. Part of this work was done while the first author was visiting the second at the University Paris 13, thanks to a funding of the MathSTIC (CNRS FR3734) research consortium.

\section{Surface triangulations, flip-graphs, and contractions}\label{PP4.sec.1}

Throughout the article, we consider a finite-type orientable surface $\Sigma$, possibly with boundary curves and with a non-empty set $P$ of marked points. In order to ensure that $P$ can be the vertex set of a triangulation of $\Sigma$, each boundary curve must contain at least one marked point. Marked points that lie in the interior of the surface we call \emph{punctures}. 

An arc is a subset of $\Sigma$ that is obtained as the continuous image of a closed interval, which is an embedding on its interior, and with the endpoints of the interval sent to $P$. We are interested in arcs up to the homotopies that preserve $P$ and we only consider arcs which are not homotopic to a point (which can only occur if the two endpoints are the same). Two arcs are said to be disjoint if they are disjoint in their interiors and a triangulation of $\Sigma$ is a set of pairwise disjoint arcs that is maximal with respect to inclusion. We call the number of arcs in a triangulation (that does not depend on the choice of the triangulation) the arc complexity of $\Sigma$ and denote it by $\kappa(\Sigma)$. 

We also consider the unoriented graph $\mathcal{F}(\Sigma)$ whose vertices are the triangulations of $\Sigma$ and whose edges link two triangulations that differ by a single arc. In other words, two triangulations $T$ and $T'$ of $\Sigma$ are the endpoints of an edge of $\mathcal{F}(\Sigma)$ when $T'$ can be obtained from $T$ by replacing an arc with a different one. This operation is called a \emph{flip} and the graph $\mathcal{F}(\Sigma)$ the \emph{flip-graph} of $\Sigma$. We will think of $\mathcal{F}(\Sigma)$ as a metric space whose elements are the triangulations of $\Sigma$ where the distance $d(T,T')$ between two triangulations $T$ and $T'$ in $\mathcal{F}(\Sigma)$ is the minimal number of flips needed to transform $T$ into $T'$.
A minimal length path between two triangulations is called geodesic. For any two triangulations $T$ and $T'$ of $\Sigma$, we denote by $\#(T,T')$ the number of geodesic paths between $T$ and $T'$ in $\mathcal{F}(\Sigma)$. For any $k$ in $\mathbb{N}$, we consider the quantity
$$
\Delta_k(\Sigma)=\max \Bigl\{ \#(T,T'):(T,T')\in\mathcal{F}(\Sigma)^2,\,d(T,T')=k\Bigr\}\mbox{.}
$$

Note that $\Delta_k(\Sigma)$ is always well-defined and finite. This is because the number of possible flips in a triangulation of $\Sigma$ is bounded by $\kappa(\Sigma)$ and hence the number of possible sequences of $k$ flips that start at a given triangulation is always at most $\kappa(\Sigma)^k$. Since a flip in a geodesic path cannot be immediately followed by the inverse flip, we have the following slightly stronger bound.

\begin{prop}\label{PP4.sec.1.prop.1}
$\displaystyle\Delta_k(\Sigma)\leq\kappa(\Sigma)(\kappa(\Sigma)-1)^{k-1}$.
\end{prop}

Throughout the article, $p$ will denote the number of punctures of $\Sigma$, $g$ its genus, and $b$ the number of its number of boundary components. It will be important to remember that $p$ only counts punctures (marked points in the interior of $\Sigma$) but ignores the number of marked points of $P$ that lie on the boundary curves of $\Sigma$. If $p$ and $b$ are simultaneously equal to $0$, then $\mathcal{F}(\Sigma)$ is empty and for this reason, we will assume that their sum is positive. If $b$ and $g$ are both equal to $0$, then $\Sigma$ is a sphere. If in addition, $p$ is equal to $1$ or $2$, there are not enough marked points to build a triangulation of $\Sigma$ (at best, one can decompose $\Sigma$ into bigons when $p$ is equal to $2$). Therefore, we assume that if $b$ and $g$ are both equal to $0$, then $p$ is at least $3$.

Let us also recall the three cases when $\mathcal{F}(\Sigma)$ is a finite graph. These three cases are marked with a triangle in Figure \ref{PP4.sec.1.fig.2}. The first case is the $3$-times punctured sphere (when $b$ and $g$ are equal to $0$ and $p$ to $3$). In that case $\mathcal{F}(\Sigma)$ is finite and the full graph is depicted in Figure \ref{PP4.sec.1.fig.1}. The second case is the disk (when $g$ and $p$ are equal to $0$ and $b$ is equal to $1$). In that case, $\mathcal{F}(\Sigma)$ is empty when there are less than $3$ marked points on the boundary of the disk and non-empty otherwise. This particular flip-graph is a graph made up of the vertices and edges of a polytope called the associahedron \cite{Lee1989,Stasheff1963a,Stasheff1963b,Tamari1951}, whose number of vertices is counted by the Catalan numbers. The third case where $\mathcal{F}(\Sigma)$ is finite is the once-punctured disk (that is when $b$ and $p$ are both equal to $1$ and $g$ to $0$) \cite{ParlierPournin2018b}. In this case $\mathcal{F}(\Sigma)$ is never empty and its size depends on the number of marked points on the boundary curve. 

We distinguish two kinds of arcs in a triangulation $T$ of $\Sigma$: if an arc in $T$ can be homotoped to the boundary of $\Sigma$ then we call that arc a \emph{boundary} arc of $T$. Any arc in $T$ that is not a boundary arc will be called an \emph{interior} arc. Observe that all the triangulations of $\Sigma$ have the same set of boundary arcs. For this reason, we will also refer to these arcs as the boundary arcs of $\Sigma$. When a boundary component $\gamma$ of $\Sigma$ contains a single point $p$ from $P$, then $p$ splits $\gamma$ into a single boundary arc $\alpha$ of $\Sigma$. In that case, $\alpha$ is twice incident to $p$ and we call this boundary arc a \emph{boundary loop}. In the remainder of the article, we denote the number of boundary arcs of $\Sigma$ by $n$. As each boundary component of $\Sigma$ contains at least one point from $P$, it gives rise to at least one boundary arc and therefore $n$ is not less than $b$. Note that when $n$ is equal to $b$, all the boundary arcs of $\Sigma$ are boundary loops.

\begin{figure}
\begin{centering}
\includegraphics[scale=1]{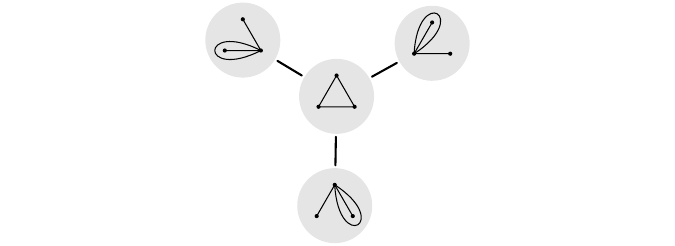}
\caption{The flip-graph of a $3$-punctured sphere.}\label{PP4.sec.1.fig.1}
\end{centering}
\end{figure}

Let us now describe an operation on $\Sigma$ and its triangulations that will be useful throughout the article. Consider a boundary arc $\alpha$ of $\Sigma$ and assume that $\alpha$ is not a boundary loop. In that case, we can build a new surface $\Sigma\contract\alpha$ by removing one of the vertices of $\alpha$ from the marked points of $\Sigma$. In other words, $\Sigma\contract\alpha$ is identical to $\Sigma$ except that its set of marked points is $P\mathord{\setminus}\{a\}$ instead of $P$, where $a$ is a vertex of $\alpha$. Note that the two surfaces obtained from this construction using, for $a$, one or the other vertex of $\alpha$ are homeomorphic, and we will identify $\Sigma\contract\alpha$ with any of these surfaces. 

Now consider a triangulation $T$ of $\Sigma$ and denote by $t$ the triangle of $T$ incident to $\alpha$. Further denote by $a$ and $a'$ the vertices of $\alpha$ as shown on the left of Figure~\ref{PP4.sec.1.fig.3}. We can obtain a triangulation $T\contract\alpha$ of $\Sigma\contract\alpha$ from $T$ as follows. First displace $a$ and $a'$ within $\alpha$ until they merge into a point that we will still denote by $a'$. During that motion, the arcs in $T$ incident to $a$ and $a'$ should be kept incident to the moving vertices as shown in Figure \ref{PP4.sec.1.fig.3}. This makes $\alpha$ disappear and transforms $t$ into the bigon that is striped in the center of Figure \ref{PP4.sec.1.fig.3}. In particular, the set of arcs resulting from that operation is not a triangulation. However, this set of arcs can be turned into the announced triangulation $T\contract\alpha$ of $\Sigma\contract\alpha$ by removing one of the arcs bounding the bigon. The only thing that remains of the triangle $t$ is the other arc bounding the bigon, that is labeled by $\beta$ on the right of Figure \ref{PP4.sec.1.fig.3}. This operation is referred to as the \emph{contraction} of $\alpha$ within either $\Sigma$ or $T$. It will be important to keep in mind that this contraction operation can only be performed when $\alpha$ is not a boundary loop.

The contraction operation was first used in \cite{Pournin2014} in order to compute the diameter of the flip-graph of a disk without punctures and later in \cite{ParlierPournin2017,ParlierPournin2018a,ParlierPournin2018b} to do the same for the modular flip-graphs of more general surfaces up to homeomorphism. It is also used more generally in \cite{PourninWang2021} to compute distances within the flip-graph of a disk without punctures. Let us recall a property of the contraction operation that is instrumental in all of these articles and will also turn up later here. This property is related to the following definition.

\begin{figure}
\begin{centering}
\includegraphics[scale=1]{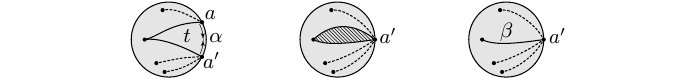}
\caption{The contraction of $\alpha$ in a triangulation of $\Sigma$.}\label{PP4.sec.1.fig.3}
\end{centering}
\end{figure}

\begin{defn}
Consider two triangulations $T$ and $T'$ of $\Sigma$ that are related by a flip. We say that this flip is incident to a boundary arc $\alpha$ of $\Sigma$ when the triangle incident to $\alpha$ is not the same in $T$ and in $T'$.
\end{defn}

In other words, a flip is incident to $\alpha$ when it exchanges the diagonals of a quadrilateral which has $\alpha$ as one of its sides. Now recall that all the triangulations of $\Sigma$ contain $\alpha$ and assume that $\alpha$ is not a boundary loop. In that case, for any two triangulations $T$ and $T'$ that are related by a flip, $T\contract\alpha$ and $T'\contract\alpha$ are two triangulations of $\Sigma\contract\alpha$ that either coincide or are still related by a flip. More precisely $T\contract\alpha$ and $T'\contract\alpha$ coincide precisely when the flip that relates $T$ and $T'$ is incident to $\alpha$. We therefore obtain the following.

\begin{lem}\label{PP4.sec.1.lem.1}
Consider a boundary arc $\alpha$ of $\Sigma$ and assume that $\alpha$ is not a boundary loop. If $\nu$ flips are incident to $\alpha$ along some geodesic path in $\mathcal{F}(\Sigma)$ between two triangulations $T$ and $T'$, then
$$
d(T,T')\geq{d(T\contract\alpha,T'\contract\alpha)}+\nu\mbox{.}
$$
\end{lem}

\section{The exponential regime}\label{PP4.sec.2}

Denote by $\Gamma$ the surface with two boundary components, genus zero, no puncture, and a single marked point on each boundary component. In this section, we prove Theorem \ref{PP4.sec.0.thm.1}: our goal is to show that, except for a few cases, $\Delta_k(\Sigma)$ is at least an exponential function of $k$ when $k$ goes to infinity. In order to do that, we are going to exhibit strongly convex subgraphs of $\mathcal{F}(\Sigma)$ that are isomorphic to the cartesian product $\mathcal{F}(\Gamma)\mathord{\times}\mathcal{F}(\Gamma)$. Indeed, we shall see that the largest possible number of geodesic paths between two vertices of that cartesian product is an exponential function of their distance. We recall that a subgraph $G$ of $\mathcal{F}(\Sigma)$ is strongly convex when the geodesic paths in $\mathcal{F}(\Sigma)$ between two vertices of $G$ are all entirely contained in $G$. Strong convexity in flip-graphs has been studied in different disguises for instance in \cite{CeballosPilaud2016,DisarloParlier2019,PourninWang2021,SleatorTarjanThurston1988}. Further recall that $\Sigma^\star$ denotes the surface obtained from $\Sigma$ by removing, on each boundary component, all the marked points except one. %Our first tool is the following lemma.

\begin{lem}\label{PP4.sec.1.5.lem.1}
$\mathcal{F}(\Sigma)$ has a strongly convex subgraph isomorphic to $\mathcal{F}(\Sigma^\star)$.
\end{lem}
\begin{proof}
Denote by $\alpha_1$ to $\alpha_b$ the boundary components of $\Sigma$. As there is at least a marked point in any of them, we can pick one of these marked points, say $a_i$ in $\alpha_i$. Cutting $\alpha_i$ at $a_i$ results in a boundary loop $\alpha_i^\star$ twice incident to $a_i$ that contains in its interior all the marked points in $\alpha_i$ but $a_i$. Now consider a sequence $\beta_1$ to $\beta_b$ arcs in $\Sigma$ such that $\beta_i$ is twice incident to $a_i$ and the portion of $\Sigma$ bounded by $\alpha_i\cup\beta_i$ is a topological disk $\Sigma_i$. By construction, $\Sigma^\star$ is homeomorphic to the surface obtained by cutting the disks $\Sigma_1$ to $\Sigma_b$ out from $\Sigma$. We will therefore identify $\Sigma^\star$ with this surface.

Now consider a sequence $T_1$ to $T_b$ of triangulations of $\Sigma_1$ to $\Sigma_b$, respectively and observe that, for any triangulation $T^\star$ of $\Sigma^\star$,
$$
\left(\bigcup_{i=1}^bT_i\right)\cup{T^\star}
$$
is a triangulation of $\Sigma$. It immediately follows that $\mathcal{F}(\Sigma)$ contains a copy of $\mathcal{F}(\Sigma^\star)$ as an induced subgraph. According to \cite{DisarloParlier2019}, that subgraph is strongly convex in $\mathcal{F}(\Sigma)$, which completes the proof.
\end{proof}

Using Lemma \ref{PP4.sec.1.5.lem.1}, we can state sufficient conditions on $\Sigma$ under whose $\mathcal{F}(\Sigma)$ has a strongly convex subgraph isomorphic to $\mathcal{F}(\Gamma)$.

\begin{prop}\label{PP4.sec.2.prop.1}
If $\Sigma$ has at least one boundary curve and at least two punctures, then $\mathcal{F}(\Sigma)$ admits a copy of $\mathcal{F}(\Gamma)$ as a strongly convex subgraph.
\end{prop}
\begin{proof}
Assume that $\Sigma$ has at least one boundary curve and at least two punctures. Denote by $\beta$ one of the boundary curves of $\Sigma$ and by $x$ one of its punctures. Pick an arc $\alpha$ in $\Sigma$ twice incident to $x$ that separates the surface into a cylinder $\Gamma'$ without punctures bounded by $\alpha$ and $\beta$ and a surface $\Sigma'$ that contains all the remaining topology of $\Sigma$ as shown on the left of Figure \ref{PP4.sec.1.fig.4}. Note that $\alpha$ is not contractible because $\Sigma'$ has at least one puncture. Therefore, $\Gamma'$ is indeed a cylinder. Pick a triangulation $T'$ of $\Sigma'$ and consider the subgraph $\mathcal{G}$ of $\mathcal{F}(\Sigma)$ induced by the triangulations that contain all the arcs in $T'$.

By construction, $\mathcal{G}$ is isomorphic to $\mathcal{F}(\Gamma')$ and according to \cite{DisarloParlier2019}, $\mathcal{G}$ is a strongly convex subgraph of $\mathcal{F}(\Sigma)$. Therefore, $\mathcal{F}(\Sigma)$ has a strongly convex subgraph isomorphic to $\mathcal{F}(\Gamma')$. This does not immediately provide the desired result because $\mathcal{F}(\Gamma')$ is not necessarily isomorphic to $\mathcal{F}(\Gamma)$. Indeed, $\beta$ can contain several marked points and, in that case, $\Gamma'$ is not homeomorphic to $\Gamma$. However, $[\Gamma']^\star$ is homeomorphic to $\Gamma$ and, by Lemma \ref{PP4.sec.1.5.lem.1}, $\mathcal{F}(\Gamma')$ has a strongly convex subgraph isomorphic to $\mathcal{F}(\Gamma)$, which proves the proposition.
\end{proof}

\begin{figure}[b]
\begin{centering}
\includegraphics[scale=1]{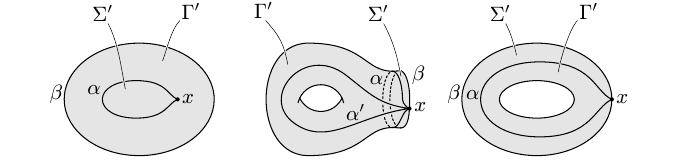}
\caption{The constructions in Propositions \ref{PP4.sec.2.prop.1}, \ref{PP4.sec.2.prop.2}, and \ref{PP4.sec.2.prop.3}.}\label{PP4.sec.1.fig.4}
\end{centering}
\end{figure}

\begin{prop}\label{PP4.sec.2.prop.2}
If $\Sigma$ has at least one boundary curve and genus at least one, then $\mathcal{F}(\Sigma)$ admits a copy of $\mathcal{F}(\Gamma)$ as a strongly convex subgraph.
\end{prop}
\begin{proof}
Assume that $\Sigma$ has at least one boundary curve $\beta$ and denote by $x$ one of the marked points in $\beta$. If in addition, $\Sigma$ has genus at least one, we can pick an arc $\alpha$ twice incident to $x$ in the interior of $\Sigma$ that separates $\Sigma$ into two surfaces one of whose is a genus one surface without punctures and a single boundary component and the other contains the remaining topology of $\Sigma$ as shown in the center of Figure~\ref{PP4.sec.1.fig.4}. Denote by $\Sigma'$ the latter surface and note that if $\alpha$ is homotopic to $\beta$, then $\Sigma'$ is just a bigon. Now pick an arc $\alpha'$ twice incident to $x$ in $\Sigma\mathord{\setminus}(\Sigma'\mathord{\setminus}\alpha))$ and denote by $\Gamma'$ the cylinder obtained by cutting $\Sigma\mathord{\setminus}(\Sigma'\mathord{\setminus}\alpha))$. Observe that $\Gamma'$ has no puncture, a single marked point on one of its boundary components, and two on the other, all three of whose are copies of $x$.

If $\alpha$ and $\beta$ are homotopic, then denote by $\mathcal{G}$ the subgraph of $\mathcal{F}(\Sigma)$ induced by the triangulations that contain $\alpha'$. Otherwise, consider a triangulation $T'$ of $\Sigma'$ and denote by $\mathcal{G}$ the subgraph of $\mathcal{F}(\Sigma)$ induced by the triangulations that contain $\alpha'$ and all the arcs in $T'$. In both cases, $\mathcal{G}$ is isomorphic to $\mathcal{F}(\Gamma')$. Moreover, according to \cite{DisarloParlier2019}, $\mathcal{G}$ is a strongly convex subgraph of $\mathcal{F}(\Sigma)$.

Finally, observe that $[\Gamma']^\star$ is homeomorphic to $\Gamma$. Hence, by Lemma \ref{PP4.sec.1.5.lem.1}, $\mathcal{G}$ admits a copy of $\mathcal{F}(\Gamma)$ as a strongly convex subgraph.%, and the result follows.
\end{proof}

\begin{prop}\label{PP4.sec.2.prop.3}
If $\Sigma$ has at least two boundary curves, then $\mathcal{F}(\Sigma)$ admits a copy of $\mathcal{F}(\Gamma)$ as a strongly convex subgraph.
\end{prop}
\begin{proof}
Assume that $\Sigma$ has two boundary curves, denote by $\beta$ one of these curves, and by $x$ one of the marked pounts in $\beta$. Pick an arc $\alpha$ twice incident to $x$ in the intrior of $\Sigma$ that cuts the surface into a cylinder $\Gamma'$ without punctures and a surface $\Sigma'$ that contains the remaining topology of $\Sigma$, as shown on the right of Figure \ref{PP4.sec.1.fig.4}. Observe that $\alpha$ and $\beta$ might be homotopic. However, if this happens, $\Sigma$ is a cylinder without punctures and $\Sigma^\star$ is homeomorphic to $\Gamma$. Therefore, the result immediately follows from Lemma \ref{PP4.sec.1.5.lem.1} in that case.

If $\alpha$ and $\beta$ are not homotopic, consider a triangulation $T'$ of $\Sigma'$. The subgraph $\mathcal{G}$ induced in $\mathcal{F}(\Sigma)$ by the triangulations that contain all the arcs in $T'$ is isomorphic to $\mathcal{F}(\Gamma')$ and, by \cite{DisarloParlier2019}, it is a strongly convex subgraph of $\mathcal{F}(\Sigma)$. Since  $\Sigma^\star$ is homeomorphic to $\Gamma$, the result follows from Lemma \ref{PP4.sec.1.5.lem.1}.
\end{proof}

Using the above three propositions, we can exhibit strongly convex subgraphs of $\mathcal{F}(\Sigma)$ isomorphic to $\mathcal{F}(\Gamma)\mathord{\times}\mathcal{F}(\Gamma)$ except in a select number of cases.

\begin{lem}\label{PP4.sec.2.lem.2}
If $\Sigma$ has at least five punctures, then $\mathcal{F}(\Sigma)$ admits a copy of the cartesian product $\mathcal{F}(\Gamma)\mathord{\times}\mathcal{F}(\Gamma)$ as a strongly convex subgraph.
\end{lem}
\begin{proof}
Assume that $\Sigma$ has at least five punctures. Consider an arc $\alpha$ in the interior of $\Sigma$ twice incident to one of these punctures and that separates $\Sigma$ into two sub-surfaces $\Sigma'$ and $\Sigma''$ one of whose is a $2$-punctured disk. The subgraph of $\mathcal{F}(\Sigma)$ induced by the triangulation that contain $\alpha$ is isomorphic to $\mathcal{F}(\Sigma')\mathord{\times}\mathcal{F}(\Sigma'')$ and according to \cite{DisarloParlier2019}, that subgraph is strongly convex in $\mathcal{F}(\Sigma)$. As both $\Sigma'$ and $\Sigma''$ have at least one boundary and at least two punctures, the lemma therefore follows from Proposition \ref{PP4.sec.2.prop.1}.
\end{proof}

\begin{lem}\label{PP4.sec.2.lem.3}
If $\Sigma$ has genus at least two and at least one puncture or at least one boundary component, then $\mathcal{F}(\Sigma)$ admits a copy of the cartesian product $\mathcal{F}(\Gamma)\mathord{\times}\mathcal{F}(\Gamma)$ as a strongly convex subgraph.
\end{lem}
\begin{proof}
 Assume that $\Sigma$ at least one puncture or one boundary component. In that case, one can pick a marked point $x$ in $\Sigma$ (which is either a puncture or a marked point in a boundary component). If in addition, $\Sigma$ has genus at least two, then we can choose an arc $\alpha$ in the interior of $\Sigma$ that is twice incident to $x$ and that separates $\Sigma$ into a genus one surface $\Sigma'$ and a positive genus surface $\Sigma''$. By construction, the subgraph induced in $\mathcal{F}(\Sigma)$ by the triangulations that admit $\alpha$ as an arc is isomorphic to $\mathcal{F}(\Sigma')\mathord{\times}\mathcal{F}(\Sigma'')$. Moreover, by \cite{DisarloParlier2019} that subgraph is strongly convex in $\mathcal{F}(\Sigma)$. As $\Sigma'$ and $\Sigma''$ each have at least one boundary component and positive genus, the result follows from  Proposition \ref{PP4.sec.2.prop.2}.
\end{proof}

\begin{lem}\label{PP4.sec.2.lem.4}
If $\Sigma$ has at least three boundary components or exactly two boundary components but at least one puncture, then $\mathcal{F}(\Sigma)$ admits a copy of $\mathcal{F}(\Gamma)\mathord{\times}\mathcal{F}(\Gamma)$ as a strongly convex subgraph.
\end{lem}
\begin{proof}
Assume that $\Sigma$ either has at least three boundary components or just two boundary components but at least one puncture. In the latter case, let $x$ denote a puncture and, in the former, a  marked point in one of the boundary components. Now pick an arc $\alpha$ in the interior of $\Sigma$ twice incident to $x$. We can require that $\alpha$ separates $\Sigma$ into two surfaces $\Sigma'$ and $\Sigma''$, both of whose have at least two boundary components. Note that $\alpha$ is a boundary component of these two surfaces and the other is one of the boundary components of $\Sigma$ that does not contain $x$. The subgraph induced in $\mathcal{F}(\Sigma)$ by the triangulations that contain $\alpha$ is isomorphic to $\mathcal{F}(\Sigma')\mathord{\times}\mathcal{F}(\Sigma'')$ and that subgraph is strongly convex in $\mathcal{F}(\Sigma)$ according to \cite{DisarloParlier2019}. By Proposition~\ref{PP4.sec.2.prop.3}, $\mathcal{F}(\Sigma')$ and $\mathcal{F}(\Sigma'')$ each admit a copy of $\mathcal{F}(\Gamma)$ as a strongly convex subgraph, which proves the lemma. 
\end{proof}

\begin{lem}\label{PP4.sec.2.lem.5}
If $\Sigma$ has positive genus and either at least two boundary components or exactly one boundary component but at least one puncture, then $\mathcal{F}(\Sigma)$ admits a copy of $\mathcal{F}(\Gamma)\mathord{\times}\mathcal{F}(\Gamma)$ as a strongly convex subgraph.
\end{lem}
\begin{proof}
 Assume that $\Sigma$ at least two boundary components or that is has a single boundary component but at least one puncture. In the former case, denote by $x$ a marked point in one of the boundary components of $\Sigma$ and in the latter case denote by $x$ one of its punctures. Under the additional assumption that $\Sigma$ has positive genus, there exists a non-separating arc $\alpha$ twice incident to $x$ in $\Sigma$. Cutting $\Sigma$ along $\alpha$ results in a surface $\Sigma'$ with at least three boundaries. By construction, the subgraph of $\mathcal{F}(\Sigma)$ induced by the triangulations that contain $\alpha$ is isomorphic to $\mathcal{F}(\Sigma')$ and according to \cite{DisarloParlier2019}, that subgraph is strongly convex in $\mathcal{F}(\Sigma)$. By Lemma \ref{PP4.sec.2.lem.4}, $\mathcal{F}(\Sigma')$ admits a copy of $\mathcal{F}(\Gamma)\mathord{\times}\mathcal{F}(\Gamma)$ as a strongly convex subgraph, which proves the lemma. 
\end{proof}

\begin{lem}\label{PP4.sec.2.lem.6}
If $\Sigma$ is a disk with at least three punctures then $\mathcal{F}(\Sigma)$ admits a copy of $\mathcal{F}(\Gamma)\mathord{\times}\mathcal{F}(\Gamma)$ as a strongly convex subgraph.
\end{lem}
\begin{proof}
Assume that $\Sigma$ is a punctured disk with at least three punctures. Consider an arc $\alpha$ twice incident to one of these punctures in the interior of $\Sigma$. We can assume that $\alpha$ cuts $\Sigma$ into a (possibly punctured) cylinder $\Sigma'$ and a $2$-punctured disk $\Sigma''$. The subgraph of $\mathcal{F}(\Sigma)$ induced by the triangulations that contain $\alpha$ is then isomorphic to $\mathcal{F}(\Sigma')\mathord{\times}\mathcal{F}(\Sigma'')$ and is is strongly convex in $\mathcal{F}(\Sigma)$ by \cite{DisarloParlier2019}. The result therefore follows from Propositions \ref{PP4.sec.2.prop.1} and \ref{PP4.sec.2.prop.3}.
\end{proof}

\begin{lem}\label{PP4.sec.2.lem.7}
If $\Sigma$ does not have a boundary, but has at least two punctures and genus exactly one, then $\mathcal{F}(\Sigma)$ admits a copy of the cartesian product $\mathcal{F}(\Gamma)\mathord{\times}\mathcal{F}(\Gamma)$ as a strongly convex subgraph.
\end{lem}
\begin{proof}
 Assume that $\Sigma$ has genus one and at least two punctures but no boundary. Consider a non-separating arc $\alpha$ in $\Sigma$ that is twice incident to one of the punctures. Cutting $\Sigma$ along $\alpha$ results in a cylinder $\Sigma'$ with at least one puncture. As in the above proofs, the subgraph of $\mathcal{F}(\Sigma)$ induced by the triangulations that admit $\alpha$ as an arc is strongly convex and isomorphic to $\mathcal{F}(\Sigma')$. Hence, the result is immediately obtained from Lemma \ref{PP4.sec.2.lem.4}.
\end{proof}

We are now ready to prove Theorem \ref{PP4.sec.0.thm.1}. The proof relies on the above lemmas and on the fact that $\mathcal{F}(\Gamma)$ is a bi-infinite simple path. More precisely, as already observed in \cite{ParlierPournin2017}, $\mathcal{F}(\Gamma)$ is the infinite  connected graph whose vertices each are incident to exactly two edges, as shown in Figure \ref{PP4.sec.1.fig.5}. In other words, $\mathcal{F}(\Gamma)$ can be modeled as the graph whose vertices are the integers in $\mathbb{Z}$ and whose edges connect any two consecutive integers. In turn, the cartesian product $\mathcal{F}(\Gamma)\mathord{\times}\mathcal{F}(\Gamma)$ is isomorphic to the graph $\mathcal{G}$ whose vertices are the points from $\mathbb{Z}^2$ and whose edges connect two points that coincide in one coordinate and differ by $1$ in the other. For any two positive integers $i$ and $j$, the distance in the graph $\mathcal{G}$ between the origin $(0,0)$ of $\mathbb{Z}^2$ and the point $(i,j)$ is precisely $i+j$ and the geodesics starting the the former point and ending in the latter are made up of edges along whose one coordinate increases by $1$ while the other coordinate does not change. As a consequence, there are
$$
\frac{(i+j)!}{i!j!}
$$
geodesics in $\mathcal{G}$ between these two points of $\mathbb{Z}^2$. As a consequence, if $k$ is any integer greater than or equal to $2$, then taking
$$
\left\{
\begin{array}{l}
i=\!\left\lfloor\frac{k}{2}\right\rfloor\!\mbox{ and}\\[\smallskipamount]
j=\!\left\lceil\frac{k}{2}\right\rceil\\
\end{array}
\right.
$$
and using Stirling's approximation of the factorial, one obtain that there are two vertices of $\mathcal{F}(\Gamma)\mathord{\times}\mathcal{F}(\Gamma)$ whose distance is equal to $k$ such that the number of geodesics between these vertices in $\mathcal{F}(\Gamma)\mathord{\times}\mathcal{F}(\Gamma)$ is greater than
$$
\left(1-\frac{1}{k}\right)\!\sqrt{\frac{2}{k\pi}}2^ke^{-\frac{1}{4(k-1)}}\mbox{.}
$$

\begin{figure}
\begin{centering}
\includegraphics[scale=1]{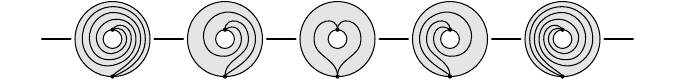}
\caption{A portion of the flip-graph of a cylinder without punctures and one marked point on each boundary.}\label{PP4.sec.1.fig.5}
\end{centering}
\end{figure}

With this estimate, we can prove the following theorem that gives an explicit exponential lower bound on $\Delta_k(\Sigma)$ and therefore implies Theorem \ref{PP4.sec.0.thm.1}.

\begin{thm}\label{PP4.sec.2.thm.1}
If $b$ and $p$ are not both equal to $0$ and $2g+p+b\geq3$, then for every integer $k$ greater than or equal to $2$,
$$
\Delta_k(\Sigma)>\!\left(1-\frac{1}{k}\right)\!\sqrt{\frac{2}{k\pi}}2^ke^{-\frac{1}{4(k-1)}}
$$
except possibly when
\begin{itemize}
\item[(i)] $\Sigma$ is the $4$-punctured sphere ($g=b=0$ and $p=4$),
\item[(ii)] $\Sigma$ is a genus one surface with a unique boundary component and no puncture ($g=b=1$ and $p=0$), or
\item[(iii)] $\Sigma$ is a $2$-punctured disk ($g=0$, $b=1$, and $p=2$).
\end{itemize}
\end{thm}

\begin{proof}
 This is a direct consequence of Lemmas \ref{PP4.sec.2.lem.2} to \ref{PP4.sec.2.lem.7} and of the discussion above on the number of geodesics between the vertices of $\mathcal{F}(\Gamma)\mathord{\times}\mathcal{F}(\Gamma)$.
\end{proof}

\begin{rem}
When $g$, $p$, and $b$ grow, $\mathcal{F}(\Sigma)$ will admit as strongly convex subgraphs cartesian products of the form $\mathcal{F}(\Gamma)^r$ where $r$ is greater than $2$. In that case, the lower bound on $\Delta_k(\Sigma)$ provided by Theorem \ref{PP4.sec.2.thm.1} can be improved. However, the asymptotic behavior of $\Delta_k(\Sigma)$ will remain exponential as we pointed out in the introduction. 
\end{rem}

\section{General estimates}\label{PP4.sec.1.5}

For any triangulation $T$ of $\Sigma$ and any $k$ in $\mathbb{N}$, we consider the set $B_\Sigma(T,k)$ of all the triangulations in $\mathcal{F}(\Sigma)$ at distance at most $k$ from $T$. Equivalently, $B_\Sigma(T,k)$ is the ball of radius $k$ centered at $T$ in $\mathcal{F}(\Sigma)$. The number of triangulations contained in $B_\Sigma(T,k)$ is bounded above by $\kappa(\Sigma)^k$ and we denote by $\Lambda_k(\Sigma)$ the largest possible number of triangulations in any such ball:
$$
\Lambda_k(\Sigma)=\max \Bigl\{ |B_\Sigma(T,k)|:T\in\mathcal{F}(\Sigma)\Bigr\}\mbox{.}
$$

We also denote by $\widetilde{\Delta}_k(\Sigma)$ the largest value of $\Delta_i(\Sigma)$ when $1\leq{i}\leq{k}$:
$$
\widetilde{\Delta}_k(\Sigma)=\mathrm{max}\bigl\{\Delta_i(\Sigma):1\leq{i}\leq{k}\bigr\}\mbox{.}
$$

The goal of this section is to prove the following theorem.

\begin{thm}\label{PP4.sec.1.5.thm.1}
For every positive integer $k$,
\begin{equation}\label{PP4.sec.1.5.thm.1.eq.1}
\Delta_k(\Sigma^\star)\leq\Delta_k(\Sigma)\leq\bigl(2\kappa(\Sigma)-4\bigr)^{(n-b)r}\Lambda_k(\Sigma^\star)^{2(r-1)}\widetilde{\Delta}_k(\Sigma^\star)^r
\end{equation}
where $r$ is equal to $\bigl(2\kappa(\Sigma)\bigr)^{n-b}$.
\end{thm}

The first inequality in (\ref{PP4.sec.1.5.thm.1.eq.1}) is easily obtained. Recall that a subgraph $G'$ of a graph $G$ is strongly convex when all the geodesic paths in $G$ between any two vertices of $G'$ remain in $G'$. Hence, we immediately get the inequality $\Delta_k(\Sigma^\star)\leq\Delta_k(\Sigma)$ stated by Theorem \ref{PP4.sec.1.5.thm.1} from Lemma \ref{PP4.sec.1.5.lem.1}.

Hence, we turn our attention to proving the second inequality in (\ref{PP4.sec.1.5.thm.1.eq.1}). As a first step, we show that when a boundary arc of $\Sigma$ is not a boundary loop, the number of flips incident to this arc along any geodesic path of length $k$ in $\mathcal{F}(\Sigma)$ can be bounded above independently from $k$. In order to show this, consider a boundary arc $\alpha$ of $\Sigma$ that is not a boundary loop and a vertex $x$ of $\alpha$. Let $\beta$ be the boundary arc of $\Sigma$ other than $\alpha$ that admits $x$ as a vertex. Consider the only triangle $t$ contained in $\Sigma$ that is incident to both $\alpha$ and $\beta$ and observe that $t$ exists precisely because $\alpha$ is not a boundary loop. Now denote by $\varepsilon$ the edge of $t$ that is distinct from $\alpha$ and $\beta$. We will denote by $\mathcal{F}_x(\Sigma)$ the subgraph induced in $\mathcal{F}(\Sigma)$ by the triangulations that admit $\varepsilon$ as an arc. Equivalently, $\mathcal{F}(\Sigma)$ is made up of all the triangulations of $\Sigma$ that do not contain any interior arc incident to $x$. Observe that the map $\psi_{x,\alpha}:\mathcal{F}_x(\Sigma)\rightarrow\mathcal{F}(\Sigma\contract\alpha)$ that sends a triangulation $T$ in $\mathcal{F}_x(\Sigma)$ to $T\contract\alpha$ is an isomorphism (and therefore, since we think of flip-graphs as metric spaces, an isometry). 

In the sequel, we denote by $\mathrm{deg}_T(x)$ the number of incidences between $x$ and the arcs of $T$. It will be important to keep in mind that any arc in $T$ that is twice incident to $x$ contributes $2$ to $\mathrm{deg}_T(x)$.

\begin{lem}\label{PP4.sec.1.5.lem.1.2}
Consider a boundary arc $\alpha$ of $\Sigma$ and a vertex $x$ of $\alpha$. If $\alpha$ is not a loop then, for any triangulation $T$ of $\Sigma$,
$$
d\bigl(T,\psi_{x,\alpha}^{-1}(T\contract\alpha)\bigr)\leq\mathrm{deg}_T(x)-2\mbox{.}
$$
\end{lem}
\begin{proof}
Consider a triangulation $T$ of $\Sigma$. Consider the triangle $t$ of $T$ that admits $\alpha$ as one of its edges and denote
%. Since $\alpha$ is not a loop, at most one of the edge of $t$ can be twice incident to $a$. Denote
$$
i(T)=
\left\{
\begin{array}{l}
\mathrm{deg}_T(x)\mbox{ if $t$ does not have an edge twice incident to $x$ and}\\
\mathrm{deg}_T(x)-1\mbox{ if $t$ has an edge twice incident to $x$.}\\
\end{array}
\right.
$$

First observe that $i(T)$ is at least $2$ because $\alpha$ is not a loop. Indeed, in that case, $x$ must be incident to two boundary arcs of $\Sigma$ (none of whose is a boundary loop). We will prove by induction on $i(T)$ that
\begin{equation}\label{PP4.sec.1.5.lem.1.2.eq.0}
d\bigl(T,\psi_{x,\alpha}^{-1}(T\contract\alpha)\bigr)\leq{i(T)-2}\mbox{.}
\end{equation}

As $i(T)$ is at most $\mathrm{deg}_T(x)$, the desired result will follow. Observe that, if $i(T)$ is equal to $2$, then $T$ belongs to $\mathcal{F}_x(\Sigma)$. Therefore, $T$ coincides with $\psi_{x,\alpha}^{-1}(T\contract\alpha)$, which proves the base case for the induction.

Now assume that $i(T)$ is at least $3$. Denote by $y$ the vertex of $\alpha$ distinct from $x$ and by $\gamma$ the edge of $t$ that is not incident to $y$. Note that $\gamma$ is a flippable interior arc of $T$. Moreover, $\gamma$ is the only edge of $t$ that can possibly be twice incident to $x$ because $y$ is distinct from $x$. Denote by $T'$ the triangulation of $\Sigma$ obtained by flipping $\gamma$ in $T$. Since $T$ and $T'$ are related by a flip incident to $\alpha$, the triangulations $T'\contract\alpha$ and $T\contract\alpha$ coincide. In particular,
\begin{equation}\label{PP4.sec.1.5.lem.1.2.eq.1}
d\bigl(T,\psi_{x,\alpha}^{-1}(T\contract\alpha)\bigr)\leq{d\bigl(T',\psi_{x,\alpha}^{-1}(T'\contract\alpha)\bigr)+1}\mbox{.}
\end{equation}

We shall now prove that
\begin{equation}\label{PP4.sec.1.5.lem.1.2.eq.1.5}
i(T')=i(T)-1\mbox{.}
\end{equation}

\begin{figure}
\begin{centering}
\includegraphics[scale=1]{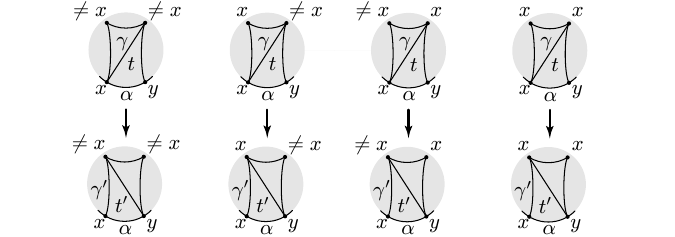}
\caption{The flip between $T$ and $T'$.}\label{PP4.sec.1.fig.6}
\end{centering}
\end{figure}

Denote by $t'$ the triangle of $T'$ incident to $\alpha$ and by $\gamma'$ the edge of $t'$ that is not incident to $y$. The flip between $T$ and $T'$ is shown in Figure \ref{PP4.sec.1.fig.6} depending on whether $\gamma$ and $\gamma'$ are twice incident to $x$ or just once. For instance, on the left of the figure, neither $\gamma$ nor $\gamma'$ are twice incident to $x$. In that case, $i(T)$ coincides with $\mathrm{deg}_T(x)$ and $i(T')$ with $\mathrm{deg}_{T'}(x)$. Hence, (\ref{PP4.sec.1.5.lem.1.2.eq.1.5}) follows from the observation that $\mathrm{deg}_{T'}(x)$ is less than $\mathrm{deg}_T(x)$ by exactly one. The situation depicted next on Figure \ref{PP4.sec.1.fig.6} is when $\gamma$ is not twice incident to $x$ but $\gamma'$ is. Therefore, $i(T)$ is equal to $\mathrm{deg}_T(x)$ and $i(T')$ to $\mathrm{deg}_{T'}(x)-1$. However, here $\mathrm{deg}_{T'}(x)$ is equal to $\mathrm{deg}_T(x)$ and we recover (\ref{PP4.sec.1.5.lem.1.2.eq.1.5}) as well. In the situation depicted third on the figure, $\gamma$ is twice incident to $x$ and $\gamma'$ only once. Hence, $i(T)$ is equal to $\mathrm{deg}_T(x)-1$ and $i(T')$ to $\mathrm{deg}_{T'}(x)$. As in this case, $\mathrm{deg}_{T'}(x)$ is less than $\mathrm{deg}_T(x)$ by two, (\ref{PP4.sec.1.5.lem.1.2.eq.1.5}) holds again. In the situation shown last, $\gamma$ and $\gamma'$ are both twice incident to $x$ and $i(T')-i(T)$ is equal to $\mathrm{deg}_{T'}(x)-\mathrm{deg}_{T'}(x)$. As $\mathrm{deg}_{T'}(x)$ is less than $\mathrm{deg}_T(x)$ by one, this proves (\ref{PP4.sec.1.5.lem.1.2.eq.1.5}), as announced. Therefore, by induction,
\begin{equation}\label{PP4.sec.1.5.lem.1.2.eq.2}
d\bigl(T',\psi_{a,\alpha}^{-1}(T'\contract\alpha)\bigr)\leq{i(T')-2}\mbox{.}
\end{equation}

Combining (\ref{PP4.sec.1.5.lem.1.2.eq.1}), (\ref{PP4.sec.1.5.lem.1.2.eq.1.5}), and (\ref{PP4.sec.1.5.lem.1.2.eq.2}) proves (\ref{PP4.sec.1.5.lem.1.2.eq.0}), as desired.
\end{proof}

\begin{thm}\label{PP4.sec.1.5.lem.2}
Consider a boundary arc $\alpha$ of $\Sigma$. If $\alpha$ is not a loop, then at most $2\kappa(\Sigma)-4$ flips are incident to $\alpha$ along any geodesic path in $\mathcal{F}(\Sigma)$.
\end{thm}
\begin{proof}
Let $x$ and $y$ be the vertices of $\alpha$ and consider two triangulations $T$ and $T'$ of $\Sigma$. As $\alpha$ is not a loop, $x$ and $y$ are distinct and
$$
\mathrm{deg}_T(x)+\mathrm{deg}_T(y)+\mathrm{deg}_{T'}(x)+\mathrm{deg}_{T'}(y)\leq4\kappa(\Sigma)\mbox{.}
$$

We can therefore assume without loss of generality that
\begin{equation}\label{PP4.sec.1.5.lem.2.eq.1}
\mathrm{deg}_T(x)+\mathrm{deg}_{T'}(x)\leq2\kappa(\Sigma)
\end{equation}
by, if needed, exchanging the labels of $x$ and $y$. Now consider a geodesic path between $T$ and $T'$ in $\mathcal{F}(\Sigma)$ and denote by $\nu$ the number of flips along it that are incident to $\alpha$. We shall prove that $\nu\leq2\kappa(\Sigma)-4$.

According to Lemma \ref{PP4.sec.1.lem.1},
$$
d(T\contract\alpha,T'\contract\alpha)+\nu\leq{d(T,T')}\mbox{.}
$$

Since $\psi_{x,\alpha}$ is an isometry, this can be rewritten as
\begin{equation}\label{PP4.sec.1.5.lem.2.eq.2}
d\bigl(\psi^{-1}_{x,\alpha}(T\contract\alpha),\psi^{-1}_{x,\alpha}(T'\contract\alpha)\bigr)+\nu\leq{d(T,T')}\mbox{.}
\end{equation}

According to Lemma \ref{PP4.sec.1.5.lem.1.2},
$$
\left\{
\begin{array}{l}
d\bigl(T,\psi_{x,\alpha}^{-1}(T\contract\alpha)\bigr)\leq\mathrm{deg}_T(x)-2\mbox{ and}\\[\smallskipamount]
d\bigl(T',\psi_{x,\alpha}^{-1}(T'\contract\alpha)\bigr)\leq\mathrm{deg}_{T'}(x)-2\mbox{.}\\
\end{array}
\right.
$$

Combining these inequalities with (\ref{PP4.sec.1.5.lem.2.eq.1}) yields
$$
d\bigl(T,\psi_{x,\alpha}^{-1}(T\contract\alpha)\bigr)+d\bigl(T',\psi_{x,\alpha}^{-1}(T'\contract\alpha)\bigr)\leq2\kappa(\Sigma)-4\mbox{.}
$$
Summing this with (\ref{PP4.sec.1.5.lem.2.eq.2}), one obtains that $2\kappa(\Sigma)-4-\nu$ is at least
$$
d\bigl(T,\psi_{x,\alpha}^{-1}(T\contract\alpha)\bigr)+d\bigl(\psi^{-1}_{x,\alpha}(T\contract\alpha),\psi^{-1}_{x,\alpha}(T'\contract\alpha)\bigr)+d\bigl(T',\psi_{x,\alpha}^{-1}(T'\contract\alpha)\bigr)-d(T,T')
$$
and it follows from the triangle inequality that $\nu\leq2\kappa(\Sigma)-4$.
\end{proof}

Recall that $n$ is the number of the boundary arcs of $\Sigma$ and $b$ is the number of its boundary components. The second ingredient for the proof of the second inequality in (\ref{PP4.sec.1.5.thm.1.eq.1}) is the following lemma whereby $\Lambda_k(\Sigma)$ only differs from $\Lambda_k(\Sigma^\star)$ by a factor that does not depend on $k$ but only on $\kappa(\Sigma)$, $n$, and $b$. 

\begin{lem}\label{PP4.sec.1.5.lem.3}
$\displaystyle\Lambda_k(\Sigma^\star)\leq\Lambda_k(\Sigma)\leq\bigl(2\kappa(\Sigma)-4\bigr)^{n-b}\Lambda_k(\Sigma^\star)$.
\end{lem}
\begin{proof}
According to Lemma \ref{PP4.sec.1.5.lem.1}, $\mathcal{F}(\Sigma)$ admits a copy of $\mathcal{F}(\Sigma^\star)$ as a strongly convex subgraph. It immediately follows that
$$
\Lambda_k(\Sigma^\star)\leq\Lambda_k(\Sigma)\mbox{.}
$$

We will prove the other inequality by induction on $n$. The base case is immediate because $\Sigma$ coincides with $\Sigma^\star$ when $n$ is equal to $b$. Assume that $n$ is greater than $b$. In that case, $\Sigma$ admits a boundary arc $\alpha$ that is not a loop. By Lemma \ref{PP4.sec.1.lem.1}, for any two triangulations $T$ and $T'$ of $\Sigma$,
$$
d(T\contract\alpha,T'\contract\alpha)\leq{d(T,T')}\mbox{.}
$$

Therefore, the map $T'\mapsto{T'\contract\alpha}$ sends $B_\Sigma(T,k)$ to a subset of $B_{\Sigma\contract\alpha}(T\contract\alpha,k)$. Now consider a triangulation $T'$ in $B_\Sigma(T,k)$. Recall that, when $\alpha$ is contracted within $T'$, then its two vertices are merged into a single vertex $a'$ and only one of the edges of the triangle $t$ of $T'$ incident to $\alpha$ remains in $T'\contract\alpha$ as shown in Figure \ref{PP4.sec.1.fig.3} where that edge is labeled $\beta$. We can recover all the triangulations $T''$ in $B_\Sigma(T,k)$ such that $T''\contract\alpha$ is equal to $T'\contract\alpha$ by reversing the process shown in Figure~\ref{PP4.sec.1.fig.3}. It suffices to pick any arc $\beta$ of $T'\contract\alpha$ incident to $a'$, to insert another arc $\beta'$ homotopic to $\beta$, thus creating a bigon bounded by $\beta\cup\beta'$ and then changing this bigon into a triangle by splitting its boundary at vertex $a'$ and connecting the two obtained copies of $a'$ with the boundary arc $\alpha$. Note that if $\beta$ is twice incident to $a'$, then the bigon can be completed into a triangle in two different ways because its boundary can be split at both of its vertices.

The number of triangulations in $B_\Sigma(T,k)$ that can be obtained by this reversed procedure is therefore the number of arcs in $T'\contract\alpha$ incident to $a'$ where the arcs twice incident to $a'$ are counted twice. This number is precisely $\mathrm{deg}_{T'\contract\alpha}(a')$ which is, in turn, not greater than $2\kappa(\Sigma\contract\alpha)$.

This shows that the map $T'\mapsto{T'\contract\alpha}$ sends at most $2\kappa(\Sigma\contract\alpha)$ triangulations from $B_\Sigma(T,k)$ to the same triangulation of $\Sigma\contract\alpha$ and as a consequence,
$$
|B_\Sigma(T,k)|\leq2\kappa(\Sigma\contract\alpha)|B_{\Sigma\contract\alpha}(T\contract\alpha,k)|\mbox{.}
$$

Picking, for $T$, a triangulation of $\Sigma$ such that $|B_\Sigma(T,k)|$ is equal to $\Lambda_k(\Sigma)$, this inequality can be rewritten as
$$
\Lambda_k(\Sigma)\leq2\kappa(\Sigma\contract\alpha)|B_{\Sigma\contract\alpha}(T\contract\alpha,k)|\mbox{.}
$$

Now observe that $\kappa(\Sigma\contract\alpha)$ is equal to $\kappa(\Sigma)-2$. As in addition $|B_{\Sigma\contract\alpha}(T\contract\alpha,k)|$ is not greater than $\Lambda_k(\Sigma\contract\alpha)$, we obtain
\begin{equation}\label{PP4.sec.1.5.lem.3.eq.1}
\Lambda_k(\Sigma)\leq\bigl(2\kappa(\Sigma)-4\bigr)\Lambda_k(\Sigma\contract\alpha)\mbox{.}
\end{equation}

However $\Sigma\contract\alpha$ has one boundary arc less than $\Sigma$. Therefore, by induction,
\begin{equation}\label{PP4.sec.1.5.lem.3.eq.2}
\Lambda_k(\Sigma\contract\alpha)\leq\bigl(2\kappa(\Sigma\contract\alpha)-4\bigr)^{n-1-b}\Lambda_k(\Sigma^\star)\mbox{.}
\end{equation}

As $\kappa(\Sigma\contract\alpha)\leq\kappa(\Sigma)$, combining (\ref{PP4.sec.1.5.lem.3.eq.1}) and (\ref{PP4.sec.1.5.lem.3.eq.2}) completes the proof.
\end{proof}

\begin{rem}
Note that, by the proof of Lemma \ref{PP4.sec.1.5.lem.3}, the upper bound it provides on $\Lambda_k(\Sigma)$ can immediately be improved to
$$
2^{n-b}\Lambda_k(\Sigma^\star)\prod_{i=1}^{n-b}\bigl(\kappa(\Sigma)-2i\bigr)
$$
when $n-b$ is greater than $1$ but we will not make use of this in the sequel.
\end{rem}

Just as in Lemma \ref{PP4.sec.1.5.lem.3}, the second inequality in (\ref{PP4.sec.1.5.thm.1.eq.1}) will be proven by induction on $n$. The third ingredient in this proof is the induction step, provided by the following lemma.

\begin{lem}\label{PP4.sec.1.5.lem.4}
If a boundary arc $\alpha$ of $\Sigma$ is not a loop, then for any positive $k$,
$$
\Delta_k(\Sigma)\leq\widetilde{\Delta}_k(\Sigma\contract\alpha)^{2\kappa(\Sigma)}\Lambda_k(\Sigma)^{2\kappa(\Sigma)}\mbox{.}
$$
\end{lem}
\begin{proof}
Consider a non-loop boundary arc $\alpha$ of $\Sigma$ and consider two triangulations $T$ and $T'$ of $\Sigma$ whose distance in $\mathcal{F}(\Sigma)$ is equal to $k$ and such that $\#(T,T')$ is equal to $\Delta_k(\Sigma)$. Denote
$$
S=\!\!\!\bigcup_{m=0}^{2\kappa(\Sigma)-1}\!\!\!B_\Sigma(T,k)^m\mbox{.}
$$

We associate an element of $S$ to each geodesic in $\mathcal{F}(\Sigma)$ between $T$ and $T'$ as follows. Denote by $T_i$ the triangulation of $\Sigma$ at distance $i$ from $T$ along such a geodesic. Denote by $\phi(1)$ to $\phi(m)$ the indices such that the flip between $T_{\phi(i)}$ and $T_{\phi(i)+1}$ is incident to $\alpha$. We assume without loss of generality that $\phi$ is an increasing function of $i$. Observe that $m$ might be equal to $0$ if there are no flip incident to $\alpha$ along the considered geodesic. Moreover, it follows from Lemma~\ref{PP4.sec.1.5.lem.2} that $m$ is less than $2\kappa(\Sigma)$. If $m=0$, then the element $s$ of $S$ that we associate with the considered geodesic is the empty set (which is the only element of $B_\Sigma(T,k)^0$). Otherwise, we associate the geodesic with
\begin{equation}\label{PP4.sec.1.5.lem.4.eq.1}
s=(T_{\phi(1)}, T_{\phi(2)},\ldots,T_{\phi(m)})\mbox{.}
\end{equation}

Observe that some elements of $S$ may not be associated to any geodesic between $T$ and $T'$. For instance, if $T$ and $T'$ do not contain the same triangle incident to $\alpha$, then the empty set canot be associated to any geodesic between $T$ and $T'$ as all of them must contain a flip incident to $\alpha$. Now let us bound the number of geodesics between $T$ and $T'$ that are associated to $s$. If $s$ is the empty set, then the triangle incident to $s$ is the same in $T$ and in $T'$. Moreover, all the triangulations along all the geodesics associated to $s$ contain that triangle. Therefore, the number of geodesics associated to $s$ is at most the number of geodesics between $T\contract\alpha$ and $T'\contract\alpha$ which, in turn is at most $\widetilde{\Delta}_k(\Sigma\contract\alpha)$ because according to Lemma \ref{PP4.sec.1.lem.1}, $d(T\contract\alpha,T'\contract\alpha)\leq{k}$.

Now, if $s$ is given by (\ref{PP4.sec.1.5.lem.4.eq.1}), then observe that along any geodesic associated with $s$, the triangle incident to $\alpha$ in $T_{\phi(i)+1}$ and $T_{\phi(i+1)}$ must be the same. In particular, the flip that relates $T_{\phi(i)}$ and $T_{\phi(i)+1}$ is prescribed by $s$ when $i<m$ and by $T'$ when $i=m$. Therefore, the number of geodesics in $\mathcal{F}(\Sigma)$ between $T$ and $T'$ that are associated to $s$ is precisely
$$
\#\bigl(T,T_{\phi(1)}\bigr)\!\left[\prod_{i=1}^{m-1}\#\bigl(T_{\phi(i)+1},T_{\phi(i+1)}\bigr)\right]\!\#\bigl(T_{\phi(m)},T'\bigr)\mbox{.}
$$

However, since the triangle incident to $\alpha$ is not modified along the portion between $T_{\phi(i)+1}$ and $T_{\phi(i+1)}$ of any geodesic associated to $s$,
$$
\#\bigl(T_{\phi(i)+1},T_{\phi(i+1)}\bigr)\leq\#\bigl(T_{\phi(i)+1}\contract\alpha,T_{\phi(i+1)}\contract\alpha\bigr)\mbox{.}
$$

According to Lemma \ref{PP4.sec.1.lem.1}, the distance between $T_{\phi(i)+1}\contract\alpha$ and $T_{\phi(i+1)}\contract\alpha$ in $\mathcal{F}(\Sigma\contract\alpha)$ is at most $k$. As a consequence
$$
\#\bigl(T_{\phi(i)+1}\contract\alpha,T_{\phi(i+1)}\contract\alpha\bigr)\leq\widetilde{\Delta}_k(\Sigma\contract\alpha)
$$
and in turn,
$$
\#\bigl(T_{\phi(i)+1},T_{\phi(i+1)}\bigr)\leq\widetilde{\Delta}_k(\Sigma\contract\alpha)
$$

By the same argument, we obtain that $\#\bigl(T,T_{\phi(1)}\bigr)$ and $\#\bigl(T_{\phi(m)},T'\bigr)$ are also both bounded above by $\widetilde{\Delta}_k(\Sigma\contract\alpha)$. Therefore, the number of geodesics associated to $s$ is at most $\widetilde{\Delta}_k(\Sigma\contract\alpha)^{m+1}$. This shows that any element of $B_\Sigma(T,k)^m$ is associated with at most this many geodesics, and we can bound the total number of geodesics in $\mathcal{F}(\Sigma)$ between $T$ and $T'$ as
$$
\Delta_k(\Sigma)\leq\!\!\!\sum_{m=0}^{2\kappa(\Sigma)-1}\!\!\!|B_\Sigma(T,k)|^{m}\widetilde{\Delta}_k(\Sigma\contract\alpha)^{m+1}
$$

The right-hand side of this inequality can be rewritten as
$$
\widetilde{\Delta}_k(\Sigma\contract\alpha)\frac{|B_\Sigma(T,k)|^{2\kappa(\Sigma)}\widetilde{\Delta}_k(\Sigma\contract\alpha)^{2\kappa(\Sigma)}-1}{|B_\Sigma(T,k)|\widetilde{\Delta}_k(\Sigma\contract\alpha)-1}
$$

Note that, under the assumption that $k$ is positive, $\widetilde{\Delta}_k(\Sigma\contract\alpha)$ is positive and $|B_\Sigma(T,k)|$ is at least $2$. In particular, non only is the denominator of the above fraction non-zero, but is is also least $\widetilde{\Delta}_k(\Sigma\contract\alpha)$. As a consequence,
$$
\Delta_k(\Sigma)\leq|B_\Sigma(T,k)|^{2\kappa(\Sigma)}\widetilde{\Delta}_k(\Sigma\contract\alpha)^{2\kappa(\Sigma)}
$$

Since $|B_\Sigma(T,k)|$ is at most $\Lambda_k(\Sigma)$, this completes the proof.
\end{proof}

We are now ready to prove Theorem \ref{PP4.sec.1.5.thm.1}.

\begin{proof}[Proof of Theorem \ref{PP4.sec.1.5.thm.1}]
By Lemma \ref{PP4.sec.1.5.lem.1}, we only need to prove the second inequality from (\ref{PP4.sec.1.5.thm.1.eq.1}). As announced, we will prove it by induction on $n$. Note that the base case, when $n$ is equal to $b$, is immediate.

Assume that $n$ is greater than $b$. In that case, $\Sigma$ has a boundary arc $\alpha$ that is not a loop. According to Lemma \ref{PP4.sec.1.5.lem.4},
$$
\Delta_k(\Sigma)\leq\Delta_i(\Sigma\contract\alpha)^{2\kappa(\Sigma)}\Lambda_k(\Sigma)^{2\kappa(\Sigma)}
$$
where $i$ denotes an integer satisfying $1\leq{i}\leq{k}$. In turn, Lemma \ref{PP4.sec.1.5.lem.3} makes it possible to bound $\Lambda_k(\Sigma)$ above in terms of $\Lambda_k(\Sigma^\star)$ and we obtain
\begin{equation}\label{PP4.sec.1.5.thm.1.eq.1.5}
\Delta_k(\Sigma)\leq\bigl(2\kappa(\Sigma)-4\bigr)^{2(n-b)\kappa(\Sigma)}\Delta_i(\Sigma\contract\alpha)^{2\kappa(\Sigma)}\Lambda_k(\Sigma^\star)^{2\kappa(\Sigma)}
\end{equation}

By induction,
\begin{equation}\label{PP4.sec.1.5.thm.1.eq.2}
\Delta_i(\Sigma\contract\alpha)\leq\bigl(2\kappa(\Sigma\contract\alpha)-4\bigr)^{(n-b-1)r'}\Lambda_i(\Sigma^\star)^{2(r'-1)}\widetilde{\Delta}_i(\Sigma^\star)^{r'}
\end{equation}
where
$$
r'=\bigl(2\kappa(\Sigma\contract\alpha)\bigr)^{n-b-1}\mbox{.}
$$

Recall that $\kappa(\Sigma\contract\alpha)\leq\kappa(\Sigma)$. As in addition $\Lambda_i(\Sigma^\star)$ and $\widetilde{\Delta}_i(\Sigma^\star)$ are nondecreasing functions of $i$, (\ref{PP4.sec.1.5.thm.1.eq.2}) yields
$$
\Delta_i(\Sigma\contract\alpha)\leq\bigl(2\kappa(\Sigma)-4\bigr)^{(n-b-1)r'}\Lambda_k(\Sigma^\star)^{2(r'-1)}\widetilde{\Delta}_k(\Sigma^\star)^{r'}\mbox{.}
$$

Combining this with (\ref{PP4.sec.1.5.thm.1.eq.1.5}), one obtains
$$
\Delta_k(\Sigma)\leq\bigl(2\kappa(\Sigma)-4\bigr)^{2(n-b+(n-b-1)r')\kappa(\Sigma)}\Lambda_k(\Sigma^\star)^{2\kappa(\Sigma)(2r'-1)}\widetilde{\Delta}_k(\Sigma^\star)^{2\kappa(\Sigma)r'}\mbox{.}
$$

Denote
$$
r=\bigl(2\kappa(\Sigma)\bigr)^{n-b}
$$
and recall that $\kappa(\Sigma\contract\alpha)\leq\kappa(\Sigma)$. Hence, $2\kappa(\Sigma)r'\leq{r}$. Moreover,
$$
2\kappa(\Sigma)(2r'-1)\leq2r-2\kappa(\Sigma)\mbox{.}
$$

As $\kappa(\Sigma)$ is positive, the right-hand side of this inequality is at most $2(r-1)$ and we can therefore bound $\Delta_k(\Sigma)$ as
\begin{equation}\label{PP4.sec.1.5.thm.1.eq.2.5}
\Delta_k(\Sigma)\leq\bigl(2\kappa(\Sigma)-4\bigr)^{2(n-b+(n-b-1)r')\kappa(\Sigma)}\Lambda_k(\Sigma^\star)^{2(r-1)}\widetilde{\Delta}_k(\Sigma^\star)^r\mbox{.}
\end{equation}

As $2\kappa(\Sigma)r'\leq{r}$, the exponent of $2\kappa(\Sigma)-4$ in (\ref{PP4.sec.1.5.thm.1.eq.2.5}) can be bounded as
\begin{equation}\label{PP4.sec.1.5.thm.1.eq.3}
2\bigl(n-b+(n-b-1)r'\bigr)\kappa(\Sigma)\leq2(n-b)\kappa(\Sigma)+(n-b-1)r
\end{equation}

However, as $\kappa(\Sigma)$ and $n-b$ are both at least $1$, we have
$$
n-b\leq\bigl(2\kappa(\Sigma)\bigr)^{n-b-1}\mbox{.}
$$

Therefore, the first term in the right-hand side of (\ref{PP4.sec.1.5.thm.1.eq.3}) is at most $r$. Since $2\kappa(\Sigma)-4$ is at least $2$ (otherwise $\Sigma$ does not admit a triangulation), the desired upper bound on $\Delta_k(\Sigma)$ follows from (\ref{PP4.sec.1.5.thm.1.eq.2.5}).
\end{proof}

\section{Polynomial cases}\label{PP4.sec.3}

Recall that $\Gamma$ is a cylinder without punctures and exactly one marked point on each boundary component. As already mentioned at the end of Section \ref{PP4.sec.2}, $\mathcal{F}(\Sigma)$ is isomorphic to a bi-infinite simple path. We obtain the following statement as an immediate consequence of that observation.

\begin{prop}\label{PP4.sec.3.prop.1}
$\Delta_k(\Gamma)=1$ and $\Lambda_k(\Gamma)=2k+1$.
\end{prop}

Now observe that, when $\Sigma$ is a cylinder without punctures, $\Sigma^\star$ is equal to $\Gamma$. Hence, Theorem \ref{PP4.sec.1.5.thm.1} and Proposition \ref{PP4.sec.3.prop.1} provide Theorem \ref{PP4.sec.0.thm.2}. More precisely, we have the following statement.

\begin{thm}\label{PP4.sec.3.thm.1}
If $\Sigma$ is a cylinder without punctures, then 
$$
\Delta_k(\Sigma)\leq\bigl(2\kappa(\Sigma)-4\bigr)^{(n-b)r}(2k+1)^{2(r-1)}
$$
where $r$ is equal to $\bigl(2\kappa(\Sigma)\bigr)^{(n-b)}$.
\end{thm}

This case corresponds to the striped disk in Figure \ref{PP4.sec.1.fig.2}. Let us turn our attention to the case of the $2$-punctured disk. For the remainder of the section, we assume that $\Sigma$ is a $2$-punctured disk. According to Theorem \ref{PP4.sec.1.5.thm.1}, we only need to consider the $2$-punctured disk $\Sigma^\star$ with only one marked point on the boundary. Denote by $x$ and $y$ the two punctures of $\Sigma^\star$ and by $z$ the marked point on its boundary. Note that there is a unique arc with vertices $x$ and $y$ up to homotopy. We denote that arc by $\varepsilon$. Denote by $\mathcal{G}$ the subgraph of $\mathcal{F}(\Sigma^\star)$ induced by the triangulations that contain $\varepsilon$ and by $\mathcal{G}'$ its subgraph induced by the triangulations that do not contain $\varepsilon$. A portion of $\mathcal{F}(\Sigma^\star)$ is shown in Figure \ref{PP4.sec.1.fig.7}. 
\begin{figure}
\begin{centering}
\includegraphics[scale=1]{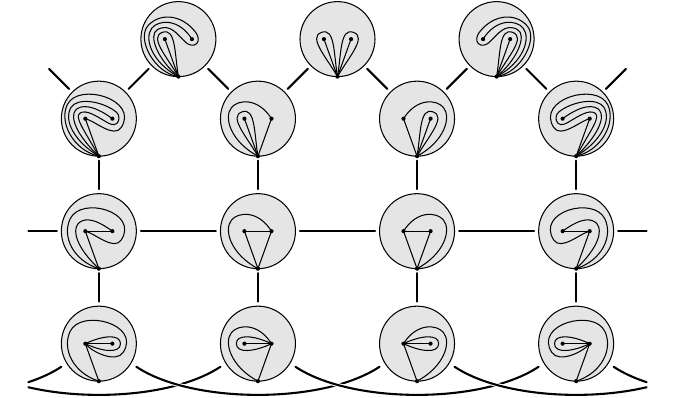}
\caption{A portion of the flip-graph of a $2$-punctured disk with one marked point in the boundary.}\label{PP4.sec.1.fig.7}
\end{centering}
\end{figure}
The seven triangulations in the two top rows of the figure are in $\mathcal{G}'$ and the eight in the two bottom rows in $\mathcal{G}$. Observe that $\mathcal{G}'$, contains two kinds of triangulations. The first kind of triangulations have two interior arcs twice incident to $z$ and two non-flippable interior arcs, one incident to $x$ and the other to $y$. Three such triangulations are shown in the top row of Figure \ref{PP4.sec.1.fig.7}. Each of them is adjacent in $\mathcal{F}(\Sigma^\star)$ to exactly two triangulations, both in $\mathcal{G}'$ and of the second kind. Four of these second kind triangulations are represented in the second row of  Figure \ref{PP4.sec.1.fig.7}. In these triangulations, there is a single interior arc twice incident to $z$. One of the punctures is incident to a non-flippable arc and the other two two flippable arcs. These triangulations are adjacent in $\mathcal{F}(\Sigma^\star)$ to two triangulations of the first kind and to exactly one triangulation in $\mathcal{G}$ since flipping the interior arc twice incident to $z$ introduces $\varepsilon$.

With this description, $\mathcal{G}'$ is, like $\mathcal{F}(\Gamma)$, a bi-infinite simple path because it is a connected infinite subgraph with all vertices of degree $2$. Moreover one can move from $\mathcal{G}'$ to $\mathcal{G}$ by a single possible flip in every other triangulation in that path, as shown in Figure \ref{PP4.sec.1.fig.7} and if two triangulations are far enough apart in $\mathcal{G}'$, it is faster to pass through $\mathcal{G}$. We get the following as a consequence.

\begin{prop}\label{PP4.sec.3.prop.2}
Consider two triangulations $T$ and $T'$ in $\mathcal{G}'$. If all the triangulations along a geodesic in $\mathcal{F}(\Sigma^\star)$ between $T$ and $T'$ belong to $\mathcal{G}'$, then the distance of $T$ and $T'$ in $\mathcal{F}(\Sigma^\star)$ is at most $6$.
\end{prop}

\begin{proof}
Consider a geodesic path in $\mathcal{F}(\Sigma^\star)$ between $T$ and $T'$. Assume that all of the triangulations along that geodesic path belong to $\mathcal{G}'$. Now recall that $\mathcal{G}'$ is a simple path whose triangulation alternate between the two types represented in the first and second rows of Figure \ref{PP4.sec.1.fig.7}. Assume for contradiction that the distance of $T$ and $T'$ in $\mathcal{F}(\Sigma^\star)$ is at least seven. In particular, there are at least eight triangulations along the considered geodesic. Therefore, one can find a sequence of seven consecutive triangulations $T_1$ to $T_7$ in that geodesic that starts and ends with triangulations from the second row of Figure \ref{PP4.sec.1.fig.7}. In fact, we can assume without loss of generality that $T_1$ is the first triangulation of the second row in that figure and $T_7$ the last. As can be seen in the figure, $T_1$ and $T_7$ are connected by a path of length five via $\mathcal{G}$, which proves the lemma.
\end{proof}

We can now prove Theorem \ref{PP4.sec.0.thm.3}.

\begin{proof}[Proof of Theorem \ref{PP4.sec.0.thm.3}]
According to Theorem \ref{PP4.sec.1.5.thm.1}, it suffices to show that $\Lambda_k(\Sigma^\star)$ and $\Delta_k(\Sigma^\star)$ are at most polynomial functions in $k$.

It is easy to see in Figure \ref{PP4.sec.1.fig.7} that the number of triangulations in any ball of radius $k$ within $\mathcal{F}(\Sigma^\star)$ is a most a linear function of $k$: roughly speaking, this figure represents $\mathcal{F}(\Sigma^\star)$ as four infinite rows. In this representation, the edges of $\mathcal{F}(\Sigma^\star)$ allow to move within a given column or from one column to the next. In the bottom row, the edges allow to skip one column, but no edge allows to skip more than one column. As a consequence, the number of triangulations in a ball of radius $k$ within $\mathcal{F}(\Sigma^\star)$ is at most $8(2k+1)$ and so is $\Lambda_k(\Sigma^\star)$.

According to Proposition \ref{PP4.sec.3.prop.2}, the length of a geodesic path that remains outside of $\mathcal{G}$ is at most $6$. However, by \cite{DisarloParlier2019}, $\mathcal{G}$ is a strongly convex subgraph of $\mathcal{F}(\Sigma^\star)$. Therefore, only the first $6$ and the last $6$ steps along any geodesic in $\mathcal{F}(\Sigma^\star)$ can possibly be outside of $\mathcal{G}$. It then suffices to prove the theorem for the geodesic paths that are entirely contained within $\mathcal{G}$. Observe that $\mathcal{G}$ is isomorphic to the flip-graph of a cylinder without punctures with a single marked point on one boundary, and two on the other. (That cylinder is the one obtained by cutting $\Sigma$ along $\varepsilon$.) It then follows from Theorem \ref{PP4.sec.0.thm.2} that the number of geodesics in $\mathcal{F}(\Sigma^\star)$ between any two triangulations in $\mathcal{G}$ is at most a polynomial function of their distance as that distance goes to infinity.
\end{proof}

\bibliography{GettingLost}
\bibliographystyle{ijmart}

\end{document}